\documentclass{amsart}
\usepackage[english]{babel}
\usepackage{amssymb,amsfonts,amsxtra}
\usepackage{dsfont,mathrsfs}
\renewcommand{\mathcal}{\mathscr}
\usepackage[all]{xypic}
\xyoption{dvips}
\usepackage{url}
\usepackage[colorlinks,plainpages,backref]{hyperref}
\usepackage[neveradjust]{paralist}

\theoremstyle{definition}
\newtheorem{ntn}{Notation}[section]
\newtheorem{dfn}[ntn]{Definition}
\theoremstyle{plain}
\newtheorem{lem}[ntn]{Lemma}
\newtheorem{prp}[ntn]{Proposition}
\newtheorem{thm}[ntn]{Theorem}
\newtheorem{cor}[ntn]{Corollary}

\theoremstyle{remark}

\newtheorem{rmk}[ntn]{Remark}
\newtheorem{exa}[ntn]{Example}
\numberwithin{equation}{section}

\newcommand{\ideal}[1]{{\left\langle#1\right\rangle}}
\newcommand{\xymat}{\SelectTips{cm}{}\xymatrix}
\newcommand{\into}{\hookrightarrow}

\newcommand{\p}{\partial}

\newcommand{\wt}{\widetilde}

\newcommand{\CC}{\mathds{C}}
\newcommand{\C}{\mathcal{C}}
\newcommand{\F}{\mathcal{F}}
\newcommand{\I}{\mathcal{I}}
\newcommand{\J}{\mathcal{J}}
\renewcommand{\L}{\mathcal{L}}
\newcommand{\M}{\mathcal{M}}
\renewcommand{\O}{\mathcal{O}}
\newcommand{\pp}{\mathfrak{p}}

\newcommand{\R}{\mathcal{R}}

\newcommand{\mm}{\mathfrak{m}}

\DeclareMathOperator{\Ass}{Ass}

\DeclareMathOperator{\depth}{depth}
\DeclareMathOperator{\Der}{Der}
\DeclareMathOperator{\End}{End}
\DeclareMathOperator{\Ext}{Ext}

\DeclareMathOperator{\Hom}{Hom}

\DeclareMathOperator{\Spec}{Spec}

\DeclareMathOperator{\RHom}{RHom}

\begin{document}
%%%%%%%%%%%%%%%%%%%%%%%%%%%%%%%%%%%%%%%%%%%%%%%%%%%%%%%%%%%%%%%%%%%%%%%%%%%%%%%

\title[Normal crossing properties of complex hypersurfaces]
{Normal crossing properties of complex hypersurfaces via logarithmic residues}

\author[M.~Granger]{Michel Granger}
\address{
M. Granger\\
Universit\'e d'Angers, D\'epartement de Math\'ematiques\\ 
LAREMA, CNRS UMR n\textsuperscript{o}6093\\ 
2 Bd Lavoisier\\ 
49045 Angers\\ 
France
}
\email{\href{granger@univ-angers.fr}{granger@univ-angers.fr}}
%\thanks{}

\author[M.~Schulze]{Mathias Schulze}
\address{M.~Schulze\\
Department of Mathematics\\
University of Kaiserslautern\\
67663 Kaiserslautern\\
Germany}
\email{\href{mailto:mschulze@mathematik.uni-kl.de}{mschulze@mathematik.uni-kl.de}}
\thanks{The research leading to these results has received funding from the People Programme (Marie Curie Actions) of the European Union's Seventh Framework Programme (FP7/2007-2013) under REA grant agreement n\textsuperscript{o} PCIG12-GA-2012-334355.}

%\date{\today}

% 32S25 Surface and hypersurface singularities
% 32A27 Local theory of residues
% 32C37 Duality theorems


\begin{abstract}
We introduce a dual logarithmic residue map for hypersurface singularities and use it to answer a question of Kyoji Saito.
Our result extends a theorem of L\^e and Saito by an algebraic characterization of hypersurfaces that are normal crossing in codimension one.
For free divisors, we relate the latter condition to other natural conditions involving the Jacobian ideal and the normalization.
This leads to an algebraic characterization of normal crossing divisors.
As a side result, we describe all free divisors with Gorenstein singular locus.
\end{abstract}

\maketitle
\tableofcontents

%%%%%%%%%%%%%%%%%%%%%%%%%%%%%%%%%%%%%%%%%%%%%%%%%%%%%%%%%%%%%%%%%%%%%%%%%%%%%%%
\section{Introduction}
%%%%%%%%%%%%%%%%%%%%%%%%%%%%%%%%%%%%%%%%%%%%%%%%%%%%%%%%%%%%%%%%%%%%%%%%%%%%%%%

Let $S$ be a complex manifold and $D$ be a (reduced) hypersurface $D$, referred to as a \emph{divisor} in the sequel. 
In the landmark paper \cite{Sai80}, Kyoji Saito introduced the sheaves of $\O_S$-modules of logarithmic differential forms and logarithmic vector fields on $S$ along $D$.
Logarithmic vector fields are tangent to $D$ at any smooth point of $D$; logarithmic differential forms have simple poles and form a complex under the usual differential.
Saito's clean algebraically flavored definition encodes deep geometric, topological, and representation-theoretic information on the singularities that is yet only partly understood.
The precise target for his theory was the Gau\ss--Manin connection on the base $S$ of the semiuniversal deformation of isolated hypersurface singularities, a logarithmic connection along the discriminant $D$.
Saito developed mainly three aspects of his logarithmic theory in \cite{Sai80}: free divisors, logarithmic stratifications, and logarithmic residues.
%Many fascinating developments grew out of Saito's paper, a few of which we highlight in the following brief overview.

A divisor is called free if the sheaf of logarithmic vector fields, or its dual, the sheaf of logarithmic $1$-forms, is a vector bundle; in particular, normal crossing divisors are free.
Not surprisingly, discriminants of isolated hypersurface singularities are free divisors (see \cite[(3.19)]{Sai80}). 
Similar results were shown for isolated complete intersection singularities (see \cite[\S6]{Loo84}) and space curve singularities (see \cite{vST95}).
Both the reflection arrangements and discriminants associated with finite unitary reflection groups are free divisors (see \cite{Ter80b}).
More recent examples include discriminants in certain prehomogeneous vector spaces (see \cite{GMS11}) whose study led to new constructions such as a chain rule for free divisors (see \cite[\S4]{BC12}).
Free divisors can be seen as the extreme case opposite to isolated singularities: unless smooth, free divisors have Cohen--Macaulay singular loci of codimension one.
The freeness property is closely related to the complement of the divisor being a $K(\pi,1)$-space (see \cite[(1.12)]{Sai80} and \cite{Del72}), although these two properties are not equivalent (see \cite{ER95}). 
Even in special cases, such as that of hyperplane arrangements, freeness is not fully understood yet. 
For instance, Terao's conjecture on the combinatorial nature of freeness for arrangements is one of the central open problems in arrangement theory.

Much less attention has been devoted to Saito's logarithmic residues, the main topic in this paper.
It was Poincar\'e who first defined a residual $1$-form of a rational differential $2$-form on $\CC^2$ (see \cite{Poi87}).
Later, the concept was generalized by de Rham and Leray to residues of closed meromorphic $p$-forms with simple poles along a smooth divisor $D$; these residues are holomorphic $(p-1)$-forms on $D$ (see \cite{Ler59}).
Residues of logarithmic differential forms along singular namely normal crossing divisors first appeared
in Deligne's construction of the mixed Hodge structure on the cohomology of smooth possibly non-compact complex varieties (see \cite[\S3.1]{Del71}); again these residues are holomorphic differential forms.
Notably, in Saito's generalization to arbitrary singular divisors $D$, the residue of a logarithmic $p$-form becomes a \emph{meromorphic} $(p-1)$-form on $D$, or on its normalization $\wt D$. 
Using work of Barlet~\cite{Bar78}, Aleksandrov linked Saito's construction to Grothendieck duality theory: the image of Saito's logarithmic residue map is the module of regular differential forms on $D$ (see \cite[\S4, Thm.]{Ale88} and \cite{Bar78}).
With Tsikh, he suggested a generalization for complete intersection singularities based on multilogarithmic differential forms depending on a choice of generators of the defining ideal (see \cite{AT01}). 
Recently, he approached the natural problem of describing the mixed Hodge structure on the complement of certain divisors in terms of logarithmic differential forms (see \cite[\S8]{Ale12}).
In Dolgachev's work, one finds a different sheaf of logarithmic differential forms which is a vector bundle exactly for normal crossing divisors and whose reflexive hull is Saito's sheaf of logarithmic differential forms (see \cite{Dol07}).
Although his approach to logarithmic residues using adjoint ideals has a similar flavor to ours, he does not reach the conclusion of our main Theorem~\ref{10} (see Remark~\ref{49b}).

\medskip

While most constructions in Saito's logarithmic theory and its generalizations have a dual counterpart, a notion of a dual logarithmic residue associated to a vector field was not known to the authors.
The main motivation and fundamental result of this article is the construction of a dual logarithmic residue (see Section~\ref{13}).
This turned out to have surprising applications including a proof of a conjecture of Saito, that was open for more than 30 years.
Saito's conjecture is concerned with comparing logarithmic residues of $1$-forms, that is, certain meromorphic functions on $\wt D$, with holomorphic functions on $\wt D$.
The latter can also be considered as \emph{weakly holomorphic} functions on $D$, that is, functions on the complement of the singular locus $Z$ of $D$, locally bounded near points of $Z$.
While any such weakly holomorphic function is the residue of some logarithmic $1$-form, the image of the residue map can contain functions which are not weakly holomorphic.
The algebraic condition of equality was related by Saito to a geometric and a topological property as follows (see \cite[(2.13)]{Sai80}).

\begin{thm}[Saito]\label{28}
For a divisor $D$ in a complex manifold $S$, consider the following conditions:
\begin{enumerate}[(A)]
\item\label{28a} the local fundamental groups of the complement $S\backslash D$ are Abelian;
\item\label{28b} in codimension one, that is, outside of an analytic subset of codimension at least $2$ in $D$, $D$ is normal crossing;
\item\label{28c} the residue of any logarithmic $1$-form along $D$ is a weakly holomorphic function on $D$.
\end{enumerate}
Then the implications \eqref{28a} $\Rightarrow$ \eqref{28b} $\Rightarrow$ \eqref{28c} hold true.
\end{thm}

Saito asked whether the the converse implications in Theorem~\ref{28} hold true.
The first one was later established by L\^e and Saito~\cite{LS84}; it generalizes the Zariski conjecture for complex plane projective nodal curves proved by Fulton and Deligne (see \cite{Ful80,Del81}).

\begin{thm}[L\^e--Saito]
The implication \eqref{28a} $\Leftarrow$ \eqref{28b} in Theorem~\ref{28} holds true.
\end{thm}

Our duality of logarithmic residues turns out to translate condition \eqref{28c} in Theorem~\ref{28} into the more familiar equality of the Jacobian ideal and the conductor ideal.
A result of Ragni Piene~\cite{Pie79} proves that such an equality forces $D$ to have only smooth components if it has a smooth normalization. 
This is a technical key point which leads to a proof of the missing implication in Theorem~\ref{28}.

\begin{thm}\label{10}
The implication \eqref{28b} $\Leftarrow$ \eqref{28c} in Theorem~\ref{28} holds true: if the residue of any logarithmic $1$-form along $D$ is a weakly holomorphic function on $D$ then $D$ is normal crossing in codimension one.
\end{thm}

\begin{rmk}\label{49b}
The logarithmic stratification of $S$ mentioned in the introduction consists of immersed integral manifolds of logarithmic vector fields along $D$ (see \cite[\S3]{Sai71}).
Contrary to what the terminology suggests, the resulting decomposition of $S$ is not locally finite, in general.
Saito attached the term \emph{holonomic} to this additional feature: a point in $S$ is holonomic if a neighborhood meets only finitely many logarithmic strata.
Along any logarithmic stratum, the pair $(D,S)$ is analytically trivial which turns holonomicity into a property of logarithmic strata.
The logarithmic vector fields are tangent to the strata of the canonical Whitney stratification; the largest codimension up to which all Whitney strata are (necessarily holonomic) logarithmic strata is called the \emph{holonomic codimension} (see \cite[p.~221]{DM91}).

Saito~\cite[(2.11)]{Sai80} proved Theorem~\ref{10} for plane curves.
If $D$ has holonomic codimension at least one, this yields the general case by analytic triviality along logarithmic strata (see \cite[\S3]{Sai80}).
Under this latter hypothesis, Theorem~\ref{10} follows also from a result of Dolgachev (see \cite[Cor.~2.2]{Dol07}).
However, for example, the equation $xy(x+y)(x+yz)=0$ defines a well-known free divisor with holonomic codimension zero.
\end{rmk}

The preceding results and underlying techniques serve to address two natural questions:
the algebraic characterization of condition~\eqref{28c} through Theorem~\ref{10} raises the question about the algebraic characterization of normal crossing divisors.
Eleonore Faber was working on this question at the same time as the results presented here were developed.
She considered freeness as a first approximation for being normal crossing and noted that normal crossing divisors satisfy an extraordinary condition:
the ideal of partial derivatives of a defining equation is radical. 
She proved the following converse implications (see \cite{Fab11,Fab12}).

\begin{thm}[Faber]\label{14}
Consider the following conditions:
\begin{enumerate}[(A)]\setcounter{enumi}{3}
\item\label{28e} at any point $p\in D$, there is a local defining equation $h$ for $D$ such that the ideal $\J_h$ of partial derivatives is radical;
\item\label{28f} $D$ is normal crossing.
\end{enumerate}
Then the following hold:
\begin{asparaenum}
\item\label{14a} if $D$ is free and satisfies condition~\eqref{28e} then the same holds for all irreducible components of $D$;
\item\label{14b} conditions~\eqref{28e} and \eqref{28f} are equivalent if $D$ is locally a plane curve or a hyperplane arrangement, or if its singular locus is Gorenstein;
\item\label{14c} condition~\eqref{28e} implies that $D$ is Euler homogeneous;
\item\label{14d} if $D$ is free, then condition~\eqref{28c} implies that $D$ is Euler homogeneous.
\end{asparaenum}
\end{thm}

Motivated by Faber's problem we prove the following result.

\begin{thm}\label{16}
Extend the list of conditions in Theorem~\ref{28} as follows:
\begin{enumerate}[(A)]\setcounter{enumi}{5}
\item\label{28d} the Jacobian ideal $\J_D$ of $D$ is radical;
\item\label{28g} the Jacobian ideal $\J_D$ of $D$ equals the conductor ideal $\C_D$ of the normalization $\wt D$.
\end{enumerate}
Then condition~\eqref{28d} implies condition~\eqref{28b}.
If $D$ is a free divisor then conditions~\eqref{28b}, \eqref{28d}, and \eqref{28g} are equivalent.
\end{thm}

\begin{rmk}\label{9}
Note that $\J_h$ is an $\O_S$-ideal sheaf depending on a choice of local defining equation whereas its image $\J_D$ in $\O_D$ is intrinsic to $D$.
However, by parts \eqref{14c} and \eqref{14d} of Theorem~\ref{14}, condition~\eqref{28e} implies condition~\eqref{28d} and equivalence holds if $D$ is free.
\end{rmk}

We obtain the following algebraic characterization of normal crossing divisors.

\begin{thm}\label{38}
For a free divisor with smooth normalization, any one of the conditions~\eqref{28a}, \eqref{28b}, \eqref{28c}, \eqref{28d}, or \eqref{28g} implies condition~\eqref{28f}.
\end{thm}

\begin{rmk}
The implication \eqref{28d} $\Rightarrow$ \eqref{28f} in Theorem~\ref{38} improves Theorem A in \cite{Fab12} (see Remark~\ref{9}), which is proved using \cite{Pie79} like in the proof of our main result.
Proposition~C in \cite{Fab12} is the implication \eqref{28c} $\Rightarrow$ \eqref{28f} in Theorem~\ref{38}, for the proof of which Faber uses our arguments.
\end{rmk}

As remarked above, free divisors are characterized by their singular loci being (empty or) maximal Cohen--Macaulay.
It is natural to ask when the singular locus of a divisor is Gorenstein.
This question is answered by the following theorem.

\begin{thm}\label{40}
A divisor $D$ has Gorenstein singular locus $Z$ of codimension $1$ if and only if $D$ is locally the product of a quasihomogeneous plane curve and a smooth space.
In particular, $D$ is locally quasihomogeneous and $Z$ is locally a complete intersection.
\end{thm}

\begin{rmk}
Theorem~\ref{40} complements a result of Kunz--Waldi \cite[Satz~2]{KW84} saying that a Gorenstein algebroid curve has Gorenstein singular locus if and only if it is quasihomogeneous.
\end{rmk}

%%%%%%%%%%%%%%%%%%%%%%%%%%%%%%%%%%%%%%%%%%%%%%%%%%%%%%%%%%%%%%%%%%%%%%%%%%%%%%%
\section{Logarithmic modules and fractional ideals}\label{30}
%%%%%%%%%%%%%%%%%%%%%%%%%%%%%%%%%%%%%%%%%%%%%%%%%%%%%%%%%%%%%%%%%%%%%%%%%%%%%%%

In this section, we review Saito's logarithmic modules, the relation of freeness and Cohen--Macaulayness of the Jacobian ideal, and the duality of maximal Cohen--Macaulay fractional ideals.
We switch to a local setup for the remainder of the article.
 
Let $D$ be a reduced effective divisor defined by $\I_D=\O_S\cdot h$ in the smooth complex analytic space germ $S=(\CC^n,0)$.
Denote by $h\colon S\to T=(\CC,0)$ a function germ generating the ideal $\I_D=\O_S\cdot h$ of $D$.
We abbreviate by
\[
\Theta_S:=\Der_\CC(\O_S)=\Hom_{\O_S}(\Omega^1_S,\O_S)
\]
the $\O_S$-module of vector fields on $S$.
Recall Saito's definition \cite[\S1]{Sai80} of the $\O_S$-modules of logarithmic differential forms and of logarithmic vector fields.

\begin{dfn}[Saito]\label{33}
We set
\begin{align*}
\Omega^p(\log D)&:=\{\omega\in\Omega^p_S(D)\mid d\omega\in\Omega_S^{p+1}(D)\}\\
\Der(-\log D)&:=\{\delta\in\Theta_S\mid dh(\delta)\in\I_D\}
\end{align*}
\end{dfn}

These modules are stalks of analogously defined coherent sheaves of $\O_S$-modules (see \cite[(1.3),(1.5)]{Sai80}).
It is obvious that each of these sheaves $\L$ is torsion free and normal, and hence reflexive (see \cite[Prop.~1.6]{Har80}).
More precisely, $\Omega^1(\log D)$ and $\Der(-\log D)$ are mutually $\O_S$-dual (see \cite[(1.6)]{Sai80}).
Normality of a sheaf $\L$ means that $\L=i_*i^*\L$ where $i\colon S\setminus Z\into S$ denotes the inclusion of the complement of the singular locus of $D$.
In the case of $\L=\Der(-\log D)$, this means that $\delta\in\Der(-\log D)$ if and only if $\delta$ is tangent to $D$ at all smooth points.
In addition, $\Omega^\bullet(\log D)$ is an exterior algebra over $\O_S$ closed under exterior differentiation and $\Der(-\log D)$ is closed under the Lie bracket.

\begin{dfn}
A divisor $D$ is called free if $\Der(-\log D)$ is a free $\O_S$-module.
\end{dfn}

The definition of $\Der(-\log D)$ can be rephrased as a short exact sequence of $\O_S$-modules
\begin{equation}\label{1}
\xymat{
0&\J_D\ar[l]&\Theta_S\ar[l]_-{dh}&\Der(-\log D)\ar[l]&0\ar[l]
}
\end{equation}
where the Jacobian ideal $\J_D$ of $D$ is defined as the Fitting ideal 
\[
\J_D:=\F^{n-1}_{\O_D}(\Omega_D^1)=\ideal{\frac{\p h}{\p x_1},\dots,\frac{\p h}{\p x_n}}\subset{\O_D}.
\]
Note that $\J_D$ is an ideal in $\O_D$ and pulls back to $\ideal{h,\frac{\p h}{\p x_1},\dots,\frac{\p h}{\p x_n}}$ in $\O_S$.
We shall consider the singular locus $Z$ of $D$ equipped with the structure defined by $\J_D$, that is,
\begin{equation}\label{21}
\O_Z:=\O_D/\J_D.
\end{equation}
Note that $Z$ might be non-reduced.
There is the following intrinsic characterization of free divisors in terms of their singular locus (see \cite[\S1 Thm.]{Ale88} or \cite[Prop.~2.4]{Ter80a}).

\begin{thm}\label{19}
The following are equivalent:
\begin{enumerate}
\item\label{19a} $D$ is a free divisor;
\item\label{19b} $\J_D$ is a maximal Cohen--Macaulay $\O_D$-module;
\item\label{19c} $D$ is smooth or $Z$ is Cohen--Macaulay of codimension one.
\end{enumerate}
\end{thm}

\begin{proof}
If $dh(\Theta_S)$ does not minimally generate $\J_D$, then $D\cong D'\times(\CC^k,0)$, $k>0$, by the triviality lemma~\cite[(3.5)]{Sai80}.
By replacing $D$ by $D'$, we may therefore assume that \eqref{1} is a minimal resolution of $\J_D$ as $\O_S$-module.
Thus, the equivalence of \eqref{19a} and \eqref{19b} is due to the Auslander--Buchsbaum formula.
By Lemma~\ref{25} below, $\J_D$ has height at least one and the implication~\eqref{19b} $\Leftrightarrow$ \eqref{19c} is proved in \cite[Satz~4.13]{HK71}.
\end{proof}

\begin{cor}\label{27}
Any $D$ is free in codimension one.
\end{cor}

\begin{proof}
By Theorem~\ref{19}, the non-free locus of $D$ is contained in $Z$ and equals
\[
\{z\in Z\mid\depth\O_{Z,z}<n-2\}\subset D.
\]
By Scheja~\cite[Satz~5]{Sch64}, this is an analytic set of codimension at least two in $D$.
\end{proof}

We denote by $Q(-)$ the total quotient ring.
Then $\M_D:=Q(\O_D)$ is the ring of meromorphic functions on $D$.

\begin{dfn}
A fractional ideal (on $D$) is a finite $\O_D$-submodule of $\M_D$ which contains a non-zero divisor.
\end{dfn}

\begin{lem}\label{25}
$\J_D$ is a fractional ideal.
\end{lem}

\begin{proof}
By assumption $D$ is reduced, so $\O_D$ satisfies Serre's condition $(R_0)$.
This means that $Z\subset D$ has codimension at least one.
In other words, $\J_D\not\subseteq\pp$ for all $\pp\in\Ass(\O_D)$ where the latter denotes the set of (minimal) associated primes of $\O_D$.
By prime avoidance, $\J_D\not\subseteq\bigcup_{\pp\in\Ass\O_D}\pp$.
But the latter is the set of zero divisors of $\O_D$ and the claim follows.
\end{proof}

\begin{prp}\label{18}
The $\O_D$-dual of any fractional ideal $\I$ is again a fractional ideal $\I^\vee=\{f\in\M_D\mid f\cdot\I\subseteq\O_D\}$.
The duality functor 
\[
-^\vee=\Hom_{\O_D}(-,\O_D)
\]
reverses inclusions. 
It is an involution on the class of maximal Cohen--Macaulay fractional ideals.
\end{prp}

\begin{proof}
See \cite[Prop.~(1.7)]{JS90a}.
\end{proof}

For lack of direct reference, we record the following consequence of the Evans--Griffith theorem.

\begin{lem}\label{39}
Let $R$ be a reduced Noetherian ring which satisfies Serre's condition $(S_2)$ and which is Gorenstein in codimension up to one.
If its normalization $\wt R$ is a finite $R$-module then it is reflexive.
\end{lem}

\begin{proof}
Let $C:=\Hom_R(\wt R,R)$ denote the conductor.
By finiteness of $\tilde R$, $\End_R(C)\subseteq\wt R$ and equality holds since $C$ is an $\wt R$-ideal.
By \cite[Exc.~1.4.19]{BH93}, with $R$ also $C$ and hence $\wt R$ satisfies $(S_2)$.
Indeed, for any $\pp\in\Spec R$,
\begin{align*}
\depth\wt R_\pp&=\depth\End_{R_\pp}(C_\pp)\\
\nonumber&\ge\min\{2,\depth C_\pp\}=\min\{2,\depth\Hom_{R_\pp}(\wt R_\pp,R_\pp)\}\\
\nonumber&\ge\min\{2,\depth R_\pp\}.
\end{align*}
By \cite[Thm.~3.6]{EG85}, this means that $\wt R$ is a reflexive $R$-module as claimed.
\end{proof}

%%%%%%%%%%%%%%%%%%%%%%%%%%%%%%%%%%%%%%%%%%%%%%%%%%%%%%%%%%%%%%%%%%%%%%%%%%%%%%%
\section{Logarithmic residues and duality}\label{13}
%%%%%%%%%%%%%%%%%%%%%%%%%%%%%%%%%%%%%%%%%%%%%%%%%%%%%%%%%%%%%%%%%%%%%%%%%%%%%%%

In this section, we develop the dual picture of Saito's residue map and apply it to find inclusion relations of certain natural fractional ideals and their duals.

Let $\pi\colon\wt D\to D$ denote the normalization of $D$.
Then $\M_D=\M_{\wt D}:=Q(\O_{\wt D})$ and $\O_{\wt D}$ is the ring of weakly holomorphic functions on $D$ (see \cite[Exc.~4.4.16.(3), Thm.~4.4.15]{JP00}).
Let
\[
\xymat{
\Omega^p(\log D)\ar[r]^-{\rho^p_D}&\Omega_D^{p-1}\otimes_{\O_D}\M_D
}
\]
be Saito's residue map \cite[\S2]{Sai80} which is defined as follows:
by \cite[(1.1)]{Sai80}, any $\omega\in\Omega^p(\log D)$ can be written as
\begin{equation}\label{4}
\omega=\frac{dh}{h}\wedge\frac{\xi}{g}+\frac{\eta}{g},
\end{equation}
for some $\xi\in\Omega_S^{p-1}$, $\eta\in\Omega^p_S$, and $g\in\O_S$, which restricts to a non-zero divisor in $\O_D$.
Then 
\begin{equation}\label{34}
\rho_D^p(\omega):=\frac{\xi}{g}\vert_D
\end{equation}
is well defined by \cite[(2.4)]{Sai80}.
We shall abbreviate $\rho_D:=\rho_D^1$ and denote its image by 
\[
\R_D:=\rho_D(\Omega^1(\log D)).
\]
Using this notation, condition~\eqref{28c} in Theorem~\ref{28}, that the residue of any $\omega\in\Omega^1(\log D)$ is weakly holomorphic, can be written as $\O_{\wt D}=\R_D$.

The following result of Saito~\cite[(2.9)]{Sai80} can be considered as a kind of approximation of our main result Theorem~\ref{10}; in fact, it shall be used in its proof.
Combined with his freeness criterion~\cite[(1.8)]{Sai80} it yields Faber's characterization of normal crossing divisors in Proposition~B of \cite{Fab12}.

\begin{thm}\label{60}
Let $D_i=\{h_i=0\}$, $i=1,\dots,k$, denote irreducible components of $D$.
Then the following conditions are equivalent:
\begin{enumerate}
\item\label{60a} $\Omega^1(\log D)=\sum_{i=1}^k\O_S\frac{dh_i}{h_i}+\Omega^1_S$;
\item\label{60b} $\Omega^1(\log D)$ is generated by closed forms;
\item\label{60d} the $D_i$ are normal and intersect transversally in codimension one;
\item\label{60c} $\R_D=\bigoplus_{i=1}^k\O_{D_i}$.
\end{enumerate}
\end{thm}

\begin{exa}
The Whitney umbrella (which is not a free divisor) satisfies the conditions in Theorem~\ref{28}, but not those in Theorem~\ref{60}. 
\end{exa}

The implication~\eqref{60d} $\Leftarrow$ \eqref{60c} in Theorem~\ref{60} will be used in the proof of Theorem~\ref{10} after a reduction to nearby germs with smooth irreducible components. 
Its proof essentially relies on parts \eqref{48b} and \eqref{48c} of the following 

\begin{exa}\label{48}\
\begin{asparaenum}

\item\label{48a} Let $D=\{xy=0\}$ be a normal crossing of two irreducible components.
Then $\frac{dx}{x}\in\Omega^1(\log D)$ and 
\[
(x+y)\frac{dx}{x}=y\frac{d(xy)}{xy}+dx-dy
\]
shows that
\[
\rho_D\left(\frac{dx}{x}\right)=\frac{y}{x+y}\Big\vert_D.
\]
On the components $D_1=\{x=0\}$ and $D_2=\{y=0\}$ of the normalization $\wt D=D_1\coprod D_2$, this function equals the constant function $1$ and $0$, respectively, and is therefore not in $\O_D$.
By symmetry, this yields
\[
\R_D=\O_{\wt D}=\O_{D_1}\times\O_{D_2}=\J_D^\vee
\]
since $\J_D=\ideal{x,y}_{\O_D}$ is the maximal ideal in $\O_D=\CC\{x,y\}/\ideal{xy}$.
This observation will be generalized in Proposition~\ref{22}.

\item\label{48b} Conversely, assume that $D_1=\{h_1=x=0\}$ and $D_2=\{h_2=x+y^m=0\}$ are two smooth irreducible components of $D$.
Consider the logarithmic $1$-form
\[
\omega=\frac{ydx-mxdy}{x(x+y^m)}=y^{1-m}\left(\frac{dh_1}{h_1}-\frac{dh_2}{h_2}\right)\in\Omega^1(\log(D_1+D_2))\subset\Omega^1(\log D).
\]
Its residue $\rho_D(\omega)\vert_{D_1}=y^{1-m}\vert_{D_1}$ has a pole along $D_1\cap D_2$ unless $m=1$.
Thus, if $\O_{\wt D}=\R_D$, then $D_1$ and $D_2$ must intersect transversally.

\item\label{48c} Assume that $D$ contains $D'=D_1\cup D_2\cup D_3$ with $D_1=\{x=0\}$, $D_2=\{y=0\}$, and $D_3=\{x-y=0\}$.
Consider the logarithmic $1$-form
\[
\omega=\frac{1}{x-y}\left(\frac{dx}{x}-\frac{dy}{y}\right)\in\Omega^1(\log D')\subset\Omega^1(\log D).
\]
Its residue $\rho_D(\omega)\vert_{D_1}=-\frac{1}{y}\vert_{D_1}$ has a pole along $D_1\cap D_2\cap D_3$ and, hence, $\O_{\wt D}\subsetneq\R_D$.

\end{asparaenum}
\end{exa}

\medskip

After these preparations, we shall now approach the construction of the dual logarithmic residue.
By definition, there is a short exact residue sequence
\begin{equation}\label{2}
\xymat{
0\ar[r]&\Omega^1_S\ar[r]&\Omega^1(\log D)\ar[r]^-{\rho_D}&\R_D\ar[r]&0.
}
\end{equation}
Applying $\Hom_{\O_S}(-,\O_S)$ to \eqref{2} gives an exact sequence
\begin{equation}\label{44}
\xymat@C-1em{
0&\Ext^1_{\O_S}(\Omega^1(\log D),\O_S)\ar[l]&\Ext^1_{\O_S}(\R_D,\O_S)\ar[l]&\Theta_S\ar[l]&\Der(-\log D)\ar[l]&0\ar[l].
}
\end{equation}
The right end of this sequence extends to the short exact sequence \eqref{1}.
For the hypersurface ring $\O_D$, the change of rings spectral sequence
\begin{equation}\label{41}
E^{p,q}_2=\Ext_{\O_D}^p(-,\Ext^q_{\O_S}(\O_D,\O_S))\underset{p}\Rightarrow\Ext_{\O_S}^{p+q}(-,\O_S)
\end{equation}
degenerates because $E^{p,q}_2=0$ if $q\neq 1$ and, hence, 
\begin{equation}\label{43}
\Ext_{\O_S}^1(-,\O_S)\cong E^{0,1}_2\cong\Hom_{\O_D}(-,\O_D)\cong-^\vee.
\end{equation}
Therefore, the second term in the sequence~\eqref{44} is $\R_D^\vee$.
This motivates the following key technical result of this paper, describing the dual of Saito's logarithmic residue.

\begin{prp}\label{22}
There is an exact sequence
\begin{equation}\label{36}
\xymat{
0&\Ext^1_{\O_S}(\Omega^1(\log D),\O_S)\ar[l]&\R_D^\vee\ar[l]&\Theta_S\ar[l]_-{\sigma_D}&\Der(-\log D)\ar[l]&0\ar[l]
}
\end{equation}
such that $\sigma_D(\delta)(\rho_D(\omega))=dh(\delta)\cdot\rho_D(\omega)$.
In particular, $\sigma_D(\Theta_S)=\J_D$ as fractional ideals.
Moreover, $\J_D^\vee=\R_D$ as fractional ideals.
\end{prp}

\begin{proof}
The spectral sequence \eqref{41} applied to $\R_D$ is associated with
\[
\RHom_{\O_S}(\Omega_S^1\into\Omega^1(\log D),h\colon\O_S\to\O_S).
\]
Expanding the double complex $\Hom_{\O_S}(\Omega_S^1\into\Omega^1(\log D),h\colon\O_S\to\O_S)$, we obtain the following diagram of long exact sequences.
\begin{equation}\label{3}\small
\xymat@C-2em{
&0\ar[d]&0\ar[d]\\
\Ext_{\O_S}^1(\R_D,\O_S)\ar[d]^-{0}&\Hom_{\O_S}(\Omega^1_S,\O_S)\ar[l]\ar[d]^-{h}&\Hom_{\O_S}(\Omega^1(\log D),\O_S)\ar[l]\ar[d]^-{h}&0\ar[l]\\
\Ext_{\O_S}^1(\R_D,\O_S)&\Hom_{\O_S}(\Omega^1_S,\O_S)\ar[l]\ar[d]&\Hom_{\O_S}(\Omega^1(\log D),\O_S)\ar[l]\ar[d]&0\ar[l]\ar[d]\\
&\Hom_{\O_S}(\Omega^1_S,\O_D)\ar[d]&\Hom_{\O_S}(\Omega^1(\log D),\O_D)\ar[l]\ar[d]&\R_D^\vee\ar[l]_-{\rho_D^\vee}\ar[d]^-\alpha&0\ar[l]\\
&0&\Ext^1_{\O_S}(\Omega^1(\log D),\O_S)\ar[l]&\Ext^1_{\O_S}(\R_D,\O_S)\ar[l]\ar[d]\\
&&&0
}
\end{equation}
We can define a homomorphism $\sigma_D$ from the upper left $\Hom_{\O_S}(\Omega^1_S,\O_S)$ to the lower right $\R_D^\vee$ by a diagram chasing process and we find that $\delta\in\Theta_S=\Hom_{\O_S}(\Omega^1_S,\O_S)$ maps to 
\[
\sigma_D(\delta)=\ideal{h\delta,\rho_D^{-1}(-)}\vert_D\in\R_D^\vee
\]
and that \eqref{36} is exact.

By comparison with the spectral sequence, one can check that $\alpha$ is the change of rings isomorphism \eqref{43} applied to $\R_D$, and that $\alpha\circ\sigma_D$ coincides with the connecting homomorphism of the top row of the diagram, which is the same as the one in \eqref{44}.
 
Let $\rho_D(\omega)\in\R_D$ where $\omega\in\Omega^1(\log D)$.
Following the definition of $\rho_D$ in \eqref{34}, we write $\omega$ in the form \eqref{4}.
Then we compute
\begin{align}
\nonumber\sigma_D(\delta)(\rho_D(\omega))&=\\\label{5}
\ideal{h\delta,\omega}\vert_D&=
dh(\delta)\cdot\frac{\xi}{g}\vert_D+h\cdot\frac{\ideal{\delta,\eta}}{g}\vert_D=
dh(\delta)\cdot\rho_D(\omega)
\end{align}
which proves the first two claims.

For the last claim, we consider the following diagram dual to \eqref{3}.
\[\small
\xymat@C-2em{
&&0\ar[d]&0\ar[d]\\
&0\ar[r]&\Hom_{\O_S}(\Theta_S,\O_S)\ar[r]\ar[d]^-{h}&\Hom_{\O_S}(\Der(-\log D),\O_S)\ar[r]\ar[d]^-{h}&\Ext_{\O_S}^1(\J_D,\O_S)\ar[d]^-{0}\\
&0\ar[r]&\Hom_{\O_S}(\Theta_S,\O_S)\ar[r]\ar[d]&\Hom_{\O_S}(\Der(-\log D),\O_S)\ar[r]\ar[d]&\Ext_{\O_S}^1(\J_D,\O_S)\\
0\ar[r]&\J_D^\vee\ar[d]^\beta\ar[r]^-{dh^\vee}&\Hom_{\O_S}(\Theta_S,\O_D)\ar[r]\ar[d]&\Hom_{\O_S}(\Der(-\log D),\O_D)\\
&\Ext_{\O_S}^1(\J_D,\O_S)\ar[r]&0
}
\]
As before, we construct a homomorphism $\rho'_D$ from the upper right $\Hom_{\O_S}(\Der(-\log D),\O_S)$ to the lower left $\J_D^\vee$ such that $\beta\circ\rho'_D$ coincides with the connecting homomorphism of the top row of the diagram, where $\beta$ is the change of rings isomorphism \eqref{43} applied to $\J_D$.
By the diagram, $\omega\in\Omega^1(\log D)=\Hom_{\O_S}(\Der(-\log D),\O_S)$ maps to
\[
\rho'_D(\omega)=\ideal{h\omega,dh^{-1}(-)}\vert_D\in \J_D^\vee
\]
which gives a short exact sequence
\begin{equation}\label{6}
\xymat{
0\ar[r]&\Omega^1_S\ar[r]&\Omega^1(\log D)\ar[r]^-{\rho'_D}&\J_D^\vee\ar[r]&0
}
\end{equation}
similar to the sequence \eqref{2}.
Using \eqref{4} and \eqref{5}, we compute 
\[
\rho'_D(\omega)(\delta(h))=
\rho'_D(\omega)(dh(\delta))=
\ideal{h\omega,\delta}\vert_D=
\rho_D(\omega)\cdot dh(\delta)=
\rho_D(\omega)\cdot \delta(h)
\]
for any $\delta(h)\in \J_D$ where $\delta\in\Theta_S$.
Hence, $\rho'_D=\rho_D$ and the last claim follows using \eqref{2} and \eqref{6}.
\end{proof}

\begin{cor}\label{37}
There is a chain of fractional ideals
\[
\J_D\subseteq\R_D^\vee\subseteq\C_D\subseteq\O_D\subseteq\O_{\wt D}\subseteq\R_D
\]
in $\M_D$ where $\C_D=\O_{\wt D}^\vee$ is the conductor ideal of $\pi$.
In particular, $\J_D\subseteq\C_D$.
\end{cor}

\begin{proof}
By Lemma~\ref{25}, $\J_D$ is a fractional ideal contained in $\R_D^\vee$ by Proposition~\ref{22}.
By \cite[(2.7),(2.8)]{Sai80}, $\R_D$ is a finite $\O_D$-module containing $\O_{\wt D}$ and, hence, a fractional ideal.
The remaining inclusions and fractional ideals are then obtained using Proposition~\ref{18}.
\end{proof}

\begin{cor}\label{24}
If $D$ is free, then $\J_D=\R_D^\vee$ as fractional ideals.
\end{cor}

\begin{proof}
If $D$ is free, then the Ext-module in the exact sequence \eqref{36} disappears and $\sigma_D$ becomes surjective.
Then the claim is part of the statement of Proposition~\ref{22}.
\end{proof}

By Corollary~\ref{37}, the inclusion $\O_{\wt D}\subset\R_D$ always holds.
For a free divisor $D$, the case of equality is translated into more familiar terms by the following corollary.

\begin{cor}\label{42}
If $D$ is free, then $\R_D=\O_{\wt D}$ is equivalent to $\J_D=\C_D$.
\end{cor}

\begin{proof}
By the preceding Corollary~\ref{24} and the last statement of Proposition~\ref{22}, the freeness of $D$ implies that $\R_D$ and $\J_D$ are mutually $\O_D$-dual.
By Lemma~\ref{39} and the definition of the conductor, the same holds true for $\O_{\wt D}$ and $\C_D$. 
The claim follows.
\end{proof}

%%%%%%%%%%%%%%%%%%%%%%%%%%%%%%%%%%%%%%%%%%%%%%%%%%%%%%%%%%%%%%%%%%%%%%%%%%%%%%%
\section{Algebraic normal crossing conditions}
%%%%%%%%%%%%%%%%%%%%%%%%%%%%%%%%%%%%%%%%%%%%%%%%%%%%%%%%%%%%%%%%%%%%%%%%%%%%%%%

In this section, we prove our main Theorem~\ref{10} settling the missing implication in Theorem~\ref{28}.
We begin with some general preparations.

\begin{lem}\label{35}
Any map $\phi\colon Y\to X$ of analytic germs with $\Omega^1_{Y/X}=0$ is an immersion.
\end{lem}

\begin{proof}
The map $\phi$ can be embedded in a map $\Phi$ of smooth analytic germs:
\[
\xymat{
Y\ar@{^(->}[r]\ar[d]^-\phi&T\ar[d]^-\Phi\\
X\ar@{^(->}[r]&S.
}
\]
Setting $\Phi_i=x_i\circ\Phi$ and $\phi_i=\Phi_i+\I_Y$ for coordinates $x_1,\dots,x_n$ on $S$ and $\I_Y$ the defining ideal of $Y$ in $T$, we can write $\Phi=(\Phi_1,\dots,\Phi_n)$ and $\phi=(\phi_1,\dots,\phi_n)$ and hence
\begin{equation}\label{26}
\Omega^1_{Y/X}=\frac{\Omega^1_Y}{\sum_{i=1}^n\O_Yd\phi_i}=\frac{\Omega^1_T}{\O_Td\I_Y+\sum_{i=1}^n\O_Td\Phi_i}.
\end{equation}
We may choose $T$ of minimal dimension so that $\I_Y\subseteq\mm_T^2$ and hence $d\I_Y\subseteq\mm_T\Omega^1_T$.
Now \eqref{26} and the hypothesis $\Omega^1_{Y/X}=0$ show that $\Omega^1_T=\sum_{i=1}^n\O_Td\Phi_i+\mm_T\Omega^1_T$ which implies that $\Omega^1_T=\sum_{i=1}^n\O_Td\Phi_i$ by Nakayama's lemma.
But then $\Phi$ and hence $\phi$ is a closed embedding as claimed.
\end{proof}

\begin{lem}\label{7}
If $\J_D=\C_D$ and $\wt D$ is smooth then $D$ has smooth irreducible components.
\end{lem}

\begin{proof}
By definition, the ramification ideal of $\pi$ is the Fitting ideal
\[
\R_\pi:=\F^0_{\O_{\wt D}}(\Omega^1_{\wt D/D}).
\]
As a special case of a result of Ragni Piene~\cite[Cor.~1, Prop.~1]{Pie79} (see also \cite[Cor.~2.7]{OZ87}), 
\[
\C_D\R_\pi=\J_D\O_{\wt D}
\]
By hypothesis, this becomes
\[
\C_D\R_\pi=\C_D
\]
since $\C_D$ is an ideal in both $\O_D$ and $\O_{\wt D}$.
As $\C_D\cong\O_{\wt D}$ (see \cite[Prop.~3.5.iii)]{MP89}), it follows that that $\R_\pi=\O_{\wt D}$ and hence that $\Omega^1_{\wt D/D}=0$.

Since $\wt D$ is normal, irreducible and connected components coincide.
By localization to a connected component $\wt D_i$ of $\wt D$ and base change to $D_i=\pi(\wt D_i)$ (see \cite[Ch.~II, Prop.~8.2A]{Har77}), we obtain $\Omega^1_{\wt D_i/D_i}=0$.
Then the normalization $\wt D_i\to D_i$ is an immersion by Lemma~\ref{35} and hence $D_i=\wt D_i$ is smooth.
\end{proof}

We are now ready to prove our main results.

\begin{proof}[Proof of Theorem~\ref{10}]
In codimension one, $D$ is free by Corollary~\ref{27} and hence $\J_D=\C_D$ by Corollary~\ref{42} and our hypothesis.
Moreover, $\wt D$ is smooth in codimension one by normality. 
By our language convention, this means that there is an analytic subset $A\subset D$ of codimension at least two such that, for $p\in D\setminus A$, $\J_{D,p}=\C_{D,p}$ and $\wt D$ is smooth above $p$.
From Lemma~\ref{7} we conclude that the local irreducible components $D_i$ of the germ $(D,p)$ are smooth. 
The hypothesis $\R_D=\O_{\wt D}$ at $p$ then reduces to the equality $\R_{D,p}=\bigoplus\O_{D_i}$.
Thus, the implication \eqref{60d} $\Leftarrow$ \eqref{60c} in Theorem~\ref{60} yields the claim. 
\end{proof}

\begin{proof}[Proof of Theorem~\ref{16}]
In order to prove that \eqref{28d} implies \eqref{28b}, we may assume that $Z$ is smooth and hence defined in $S$ by two of the generators $h,\p_1(h),\dots,\p_n(h)$.
Then the triviality lemma~\cite[(3.5)]{Sai80} shows that $D$ is either smooth or $D\cong C\times(\CC^{n-2},0)$ and $C\subset(\CC^2,0)$ a plane curve. 
We may therefore reduce to the case of a plane curve.
Then the Mather--Yau theorem~\cite{MY82} applies (see \cite[Prop.~9]{Fab12} for details).

Now assume that $D$ is free and normal crossing in codimension one. 
By the first assumption and Theorem~\ref{19}, $Z$ is Cohen--Macaulay of codimension one and, in particular, satisfies Serre's condition $(S_1)$.
By the second assumption, $Z$ also satisfies Serre's condition $(R_0)$.
Then $Z$ is reduced, and hence $\J_D$ is radical, by Serre's reducedness criterion.
This proves that \eqref{28b} implies \eqref{28d} for free $D$.

The last equivalence then follows from Theorems~\ref{28} and \ref{10} and Corollary~\ref{42}.
\end{proof}

\begin{proof}[Proof of Theorem~\ref{38}]
By Theorems~\ref{28}, \ref{10}, and \ref{16}, we may assume that $\J_D=\C_D$.
Then Lemma~\ref{7} shows that the irreducible components $D_i=\{h_i=0\}$, $i=1,\dots,m$, of $D$ are smooth, and hence normal.
It follows that 
\[
\R_D=\O_{\wt D}=\bigoplus_{i=1}^m\O_{\wt D_i}=\bigoplus_{i=1}^m\O_{D_i}.
\]
By the implication \eqref{60a} $\Leftarrow$ \eqref{60c} in Theorem~\ref{60}, this is equivalent to 
\begin{equation}\label{51}
\Omega^1(\log D)=\sum_{i=1}^m\O_S\frac{dh_i}{h_i}+\Omega^1_S.
\end{equation}
On the other hand, Saito's criterion \cite[(1.8) Thm. i)]{Sai80} for freeness of $D$ reads
\begin{equation}\label{52}
\bigwedge^n\Omega^1(\log D)=\Omega^n_S(D).
\end{equation}
Combining \eqref{51} and \eqref{52}, it follows immediately that $D$ is normal crossing (see also \cite[Prop.~B]{Fab12}):

Considered as an $\O_S$-module and modulo $\Omega_S^n$, the left-hand side of \eqref{52} is, due to \eqref{51}, generated by expressions 
\begin{gather}\label{53}
\frac{d h_{i_1}\wedge\dots\wedge d h_{i_k}\wedge dx_{j_1}\wedge\dots\wedge dx_{j_{n-k}}}{h_{i_1}\cdots h_{i_k}},\\ 
\nonumber1\le i_1<\cdots<i_k\le m, \quad 1\le j_1<\cdots<j_{n-k}\le n\ge k
\end{gather}
whereas the right-hand side of \eqref{52} is generated by
\[
\frac{d x_1\wedge\dots\wedge d x_n}{h_1\cdots h_m}.
\]
In order for an instance of \eqref{53} to attain the denominator of this latter expression, it must satisfy $k=m$ and, in particular, $m\le n$.
Further, comparing numerators, $d h_1\wedge\dots\wedge d h_m\wedge dx_{j_1}\wedge\dots\wedge dx_{j_{n-m}}$ must be a unit multiple of $d x_1\wedge\dots\wedge d x_n$.
In other words, choosing $i_1,\dots,i_m$ such that $\{i_1,\dots,i_m,j_1,\dots,j_{n-m}\}=\{1,\dots,n\}$, 
\[
\frac{\p(h_1,\dots,h_m)}{\p(x_{i_1},\dots,x_{i_m})}\in\O_S^*.
\]
By the implicit function theorem, $h_1,\dots,h_m,x_{j_1},\dots,x_{j_{n-m}}$ is then a coordinate system and hence $D$ is a normal crossing divisor as claimed.
\end{proof}

%%%%%%%%%%%%%%%%%%%%%%%%%%%%%%%%%%%%%%%%%%%%%%%%%%%%%%%%%%%%%%%%%%%%%%%%%%%%%%%
\section{Gorenstein singular locus}
%%%%%%%%%%%%%%%%%%%%%%%%%%%%%%%%%%%%%%%%%%%%%%%%%%%%%%%%%%%%%%%%%%%%%%%%%%%%%%%

In this section, we describe the canonical module of the singular locus $Z$ in terms of the module $\R_D$ of logarithmic residues.
Then, we prove Theorem~\ref{40}.

The complex of logarithmic differential forms along $D$ relative to the map $h\colon S\to T:=(\CC,0)$ defining $D$ is defined as $\Omega^\bullet(\log h):=\Omega^\bullet(\log D)/\frac{dh}h\wedge\Omega^{\bullet-1}(\log D)$ (see \cite[\S22]{GMS09} or \cite[Def.~2.7]{DS12}).

\begin{prp}\label{20}
The $\O_Z$-module $\R_h:=\R_D/\O_D$ fits into an exact square
\[
\xymat{
&0\ar[d]&0\ar[d]&0\ar[d]\\
0\ar[r]&\O_S\ar[r]^-h\ar[d]^-{dh}&\O_S\ar[r]\ar[d]^-{\frac{dh}{h}}&\O_D\ar[r]\ar[d]&0\\
0\ar[r]&\Omega^1_S\ar[r]\ar[d]&\Omega^1(\log D)\ar[r]^-{\rho_D}\ar[d]&\R_D\ar[r]\ar[d]&0\\
0\ar[r]&\Omega^1_{S/T}\ar[r]\ar[d]&\Omega^1(\log h)\ar[r]\ar[d]&\R_h\ar[r]\ar[d]&0\\
&0&0&0
}.
\]
If $D$ is free then $Z$ is Cohen--Macaulay with canonical module $\omega_Z=\R_h$; in particular, $Z$ is Gorenstein if and only if $\R_D$ is generated by $1$ and one additional generator.
\end{prp}

\begin{proof}
We set $\omega_\emptyset:=0$ in case $D$ is smooth and assume $Z\ne\emptyset$ in the following.
The exact square arises from the residue sequence \eqref{2} using the Snake lemma.
Dualizing the short exact sequence
\[
\xymat{
0\ar[r]&\J_D\ar[r]&\O_D\ar[r]&\O_Z\ar[r]&0,
}
\]
one computes
\begin{equation}\label{29}
\Ext_{\O_D}^1(\O_Z,\O_D)=\J_D^\vee/\O_D=\R_D/\O_D=\R_h.
\end{equation}
which shows that $\R_h$ is an $\O_Z$-module.
If $D$ is free then, by Theorem~\ref{19}, $Z$ is Cohen--Macaulay of codimension one and $\omega_Z=\R_h$ by \eqref{29}.
\end{proof}

Part~\eqref{56e} of the following proposition can also be found in Faber's thesis (see \cite[Prop.~1.29]{Fab11}).

\begin{prp}\label{56}
Let $D$ be a free divisor.  
Then the following statements hold true:
\begin{asparaenum} 
\item\label{56e} $\frac{dh}h$ is part of an $\O_S$-basis of $\Omega^1(\log D)$ if and only if $D$ is Euler homogeneous;
\item\label{56d} $\frac{dh}h$ is part of an $\O_S$-basis of $\Omega^1(\log D)$ if and only if $1$ is part of a minimal set of $\O_D$-generators of $\R_D$;
\item\label{56g} $\R_D$ is a cyclic $\O_D$-module if and only if $D$ is smooth.
\end{asparaenum}
\end{prp}

\begin{proof}\
\begin{asparaenum} 

\item This is immediate from the existence of a dual basis and the fact that $\chi\in\Der(-\log D)$ is an Euler vector field exactly if $\ideal{\chi,\frac{dh}h}=\frac{\chi(h)}h=1$. 
Then $\chi$ can be chosen as a member of some basis.

\item We may assume that $D\cong D'\times S'$ with $S'=(\CC^r,0)$ implies $r=0$.
Indeed, a basis of $\Omega^1(\log D)$ is the union of bases of $\Omega^1(\log D')$ and $\Omega_{S'}^1$.
This assumption is equivalent to $\Der(-\log D)\subseteq\mm_S\Theta_S$, where $\mm_S$ denotes the maximal ideal of $\O_S$. 
Dually, this means that no basis element of $\Omega^1(\log D)$ can lie in $\Omega_S^1$.

Consider a basis $\omega_1,\dots,\omega_n$ of $\Omega^1(\log D)$ with $\omega_1=\frac{dh}h$ and set $\rho_i:=\rho_D(\omega_i)$ for $i=1,\dots,n$.
If $\rho_1=1$ is not a member of some minimal set of generators of $\R_D$, then $\rho_1=\sum_{i=2}^na_i\rho_i$ with $a_i\in\mm_S$. 
Thus, the form $\omega_1':=\omega_1-\sum _{i=2}^na_i\omega_i$ can serve as a replacement for $\omega_1$ in the basis. 
But, by construction, $\rho_D(\omega'_1)=0$ which implies that $\omega'_1\in\Omega_S$ by \eqref{2}.
This contradicts our assumption on $D$.
 
Conversely, suppose that $\frac{dh}h$ is not a member of any $\O_S$-basis $\omega_1,\dots,\omega_n$ of $\Omega^1(\log D)$.
Then, by Nakayama's lemma, there are $a_i\in\mm_S$ such that $\frac{dh}h=\sum_{i=1}^na_i\omega_i$.
Applying $\rho_D$, this gives $1=\rho_D(\frac{dh}h)=\sum_{i=1}^na_i\rho_i\in\mm_D\R_D$.
Again by Nakayama's lemma, this means that $1$ is not a member of any minimal set of $\O_D$-generators of $\R_D$.

\item If $\R_D\cong\O_D$, then also $\J_D\cong\O_D$ by Corollary~\ref{24} then $D$ must be smooth by Lipman's criterion~\cite{Lip69}.
Alternatively, it follows that $\J_h+\O_S\cdot h=\O_S$ and hence also $\J_h=\O_S$ (see Remark~\ref{9}); so $D$ is smooth by the Jacobian criterion.

\end{asparaenum} 
\end{proof}

As observed by Faber~\cite[Rmk.~54]{Fab12}, condition~\eqref{56d} of Proposition~\ref{56} is satisfied given $\R_D=\O_{\wt D}$ (see Theorem~\ref{14}.\eqref{14d}).

\begin{proof}[Proof of Theorem~\ref{40}]
Assume that $Z$ is Gorenstein of codimension one in $D$.
The preimage $\J'_D$ of $\J_D$ in $\O_S$ is then a Gorenstein ideal of height two.
As such, it is a complete intersection ideal by a theorem of Serre (see \cite[Cor.~21.20]{Eis95}), and hence generated by two of the generators $h,\p_1(h),\dots,\p_n(h)$.
As in the proof of Theorem~\ref{16}, the triviality lemma~\cite[(3.5)]{Sai80} shows that $D$ is either smooth or $D\cong C\times(\CC^{n-2},0)$ and $C\subset(\CC^2,0)$ a plane curve.
Then also $C$ has Gorenstein singular locus and is hence quasihomogeneous by \cite{KW84}.
Alternatively, the last implication follows from Propositions~\ref{20} and \ref{56} using that quasihomogeneity of $C$ follows from Euler homogeneity of $C$ by Saito's quasihomogeneity criterion~\cite{Sai71} for isolated singularities.
Finally, $D$ is quasihomogeneous and $Z$ is a complete intersection.
The converse is trivial.
\end{proof}

%%%%%%%%%%%%%%%%%%%%%%%%%%%%%%%%%%%%%%%%%%%%%%%%%%%%%%%%%%%%%%%%%%%%%%%%%%%%%%%
\section*{Acknowledgements}
%%%%%%%%%%%%%%%%%%%%%%%%%%%%%%%%%%%%%%%%%%%%%%%%%%%%%%%%%%%%%%%%%%%%%%%%%%%%%%%

We are grateful to Eleonore Faber and David Mond for helpful discussions and comments, and to Ragnar-Olaf Buchweitz for pointing out the statement of Lemma~\ref{39} and for outlining a proof.
The foundation for this paper was laid during a ``Research in Pairs'' stay at the ``Mathematisches Forschungsinstitut Oberwolfach'' in summer 2011.
We thank the anonymous referee for careful reading and helpful comments.

%%%%%%%%%%%%%%%%%%%%%%%%%%%%%%%%%%%%%%%%%%%%%%%%%%%%%%%%%%%%%%%%%%%%%%%%%%%%%%%
\bibliographystyle{amsalpha}
\bibliography{dlr}
%%%%%%%%%%%%%%%%%%%%%%%%%%%%%%%%%%%%%%%%%%%%%%%%%%%%%%%%%%%%%%%%%%%%%%%%%%%%%%%
\end{document}